\documentclass[11pt]{amsart}
\usepackage{amsmath,cite}
\usepackage[english,  activeacute]{babel}
\usepackage[latin1]{inputenc}
\usepackage{amssymb}
\usepackage{amsthm}
\usepackage{graphics,graphicx}
\usepackage{array}
\usepackage{a4wide}
\allowdisplaybreaks
\usepackage{color, url}
\usepackage{float}

\theoremstyle{plain}
\newtheorem{thm}{Theorem}[section]

\newtheorem{lem}[thm]{Lemma}

\newtheorem{ex}[thm]{Example}

\numberwithin{equation}{section}

\begin{document}
\title[Determinant identities for Toeplitz-Hessenberg]{Determinant identities for Toeplitz-Hessenberg matrices with tribonacci number entries}
\author[T. Goy]{Taras Goy}
\author[M. Shattuck]{Mark Shattuck}

\begin{abstract}
In this paper, we evaluate determinants of some families of Toeplitz--Hessenberg matrices having tribonacci number entries. These determinant formulas may also be expressed equivalently as identities that involve sums of products of multinomial coefficients and tribonacci numbers. In particular, we establish a connection between the tribonacci  and the Fibonacci and Padovan sequences via Toeplitz--Hessenberg determinants.  We then obtain, by combinatorial arguments, extensions of our determinant formulas in terms of generalized tribonacci sequences satisfying a recurrence of the form $T_n^{(r)}=T_{n-1}^{(r)}+T_{n-2}^{(r)}+T_{n-r}^{(r)}$ for $n \geq r$, with the appropriate initial conditions, where $r \geq 3$ is arbitrary.
\end{abstract}
\subjclass[2010]{ Primary: 05A19; Secondary: 11B39, 15B05.}
\keywords{tribonacci numbers, Toeplitz-Hessenberg matrix, determinant, multinomial coefficient}

\date{\today}

\maketitle

\section{Introduction}

Let $\{F_n\}_{n\geq0}$ denote the \emph{Fibonacci sequence} (sequence A000045 in the OEIS \cite{Sloane}) given by $F_0 = 0$, $F_1 = 1$, and $F_{n} = F_{n-1} + F_{n-2}$ for $n \geq2$.  Some of the best known of the many analogues of the Fibonacci numbers include the \emph{tribonacci} and \textit{Padovan} numbers, denoted here by $T_n$ and $P_n$ (see entries A000073 and A000931 in \cite{Sloane}, respectively).
The tribonacci and Padovan sequences are defined respectively by the following recurrence relations for $n\geq3$:
\begin{align*}
T_{n}=T_{n-1}+T_{n-2}+T_{n-3},\qquad T_0=T_1=0,\, T_2 = 1,\\
P_{n}=P_{n-2}+P_{n-3},\qquad P_0=1,\, P_1=P_2=0.
\end{align*}
The tribonacci numbers represent the $k=3$ case of the Fibonacci $k$-step numbers \cite{Dresden} and are given equivalently by the explicit formula \cite{Barry}
$$T_{n}=
\sum_{i=0}^{\left\lfloor{n}/{2}\right\rfloor-1}\sum_{j=0}^i{i\choose j}\,{n-2-i-j \choose i}\,,\qquad n\geq2,$$
where $\lfloor\alpha\rfloor$ denotes the floor of $\alpha$.

Tribonacci numbers have a long history and have been extensively studied. First introduced by Agro\-nomof \cite{Agronomof} in 1914, the name \emph{tribonacci} being later coined by Feinberg \cite{Feinberg}, these numbers have been considered by many others since.  Examples of some recent work that involve tribonacci numbers and their various generalizations include \cite{Anan,Cereceda,Choi1,Choi2,Feng,Irmak,Kilic,Kuha,Yilmaz, Zatorsky}.
For instance, Choi \cite{Choi1} found a tribonacci triangle which is analogous to Pascal's and also investigated an efficient method of computing the $n$th tribonacci number using matrices.  Kili\c c \cite{Kilic} found some identities and generating matrices for the sequences $\{T_n\}$ and $\{T_{4n}\}$, along with their sums. Irmak and Alp \cite{Irmak}  later gave a recurrence relation for the tribonacci numbers with subscripts in arithmetic progression, $\{T_{rn+s}\}$ for $0\leq s<r$, and found sums of $\{T_{rn}\}$ for arbitrary $r$ via matrix methods. In \cite{Feng}, Feng  derived various recurrence relations for the tribonacci numbers and their sums and obtained some related identities by use of companion and generating matrices.  Kuhapatanakul and Sukruan \cite{Kuha} derived an explicit formula for the generalized tribonacci polynomials with negative subscripts possible.  In this paper, we find some new formulas for the tribonacci numbers which can be expressed equivalently in terms of certain determinant or multi-sum expressions.

The organization of this paper is as follows. The next section concerns preliminaries, while in the third, we find new relations involving the tribonacci sequence which arise as determinants of certain families of Toeplitz--Hessenberg matrices.  We remark that some of the results from this section were announced without proofs in \cite{Goy}.  Next, we state multi-sum versions of these relations involving products of multinomial coefficients and powers of tribonacci numbers.  In the final section, we extend by combinatorial arguments the preceding determinant formulas to generalized tribonacci sequences satisfying a recurrence of the form $T_n^{(r)}=T_{n-1}^{(r)}+T_{n-2}^{(r)}+T_{n-r}^{(r)}$ where $r \geq 3$.  It should be noted that while one of several possible extensions of the tribonacci number sequence might be considered, it is this one that seems to yield the most generalizations of the determinant formulas from the third section.

\section{Toeplitz--Hessenberg matrices and determinants}

A \emph{Toeplitz--Hessenberg matrix} is an $n\times n$ matrix having the form
\begin{equation}\label{THmatrix}
M_n(a_0;  a_1,\ldots,a_n)=\left[\begin{array}{ccccccc}
a_1        & a_0      & 0   &  \cdots       & 0      & 0 \\
a_2        & a_1       & a_0   & \cdots        & 0      & 0 \\
a_3        & a_2       & a_1      & \cdots        & 0      & 0 \\
\cdots     & \cdots    & \cdots & \ddots & \cdots & \cdots \\
a_{n-1}   & a_{n-2}  & a_{n-3} & \cdots  & a_1    & a_0 \\
a_{n}     & a_{n-1}& a_{n-2}   & \cdots & a_2    & a_1\\
\end{array}\right],
\end{equation}
where $a_0\ne 0$ and at least one $a_k\ne 0$ for $k>0$.  Toeplitz--Hessenberg matrices are often encountered in various applications of science and engineering (see, for example,  \cite{Merca} and references therein).

Repeated expansion along the first row yields the recurrence
\begin{equation}\label{Det-Hess}
\det(M_n)=\sum_{k=1}^{n}(-a_0)^{k-1}a_k\det(M_{n-k}), \qquad n \geq 1,
\end{equation}
where $\det(M_0)=1$, by convention.

The following result, which provides a multinomial expansion of  $\det(M_n)$, is known as Trudi's formula, the  $a_0 = 1$ case of which is called Brioschi's formula \cite{Muir}.
\begin{lem} Let $n$ be a positive integer. Then
\begin{equation} \label{Trudi}
\det(M_n)=\sum_{(s_1,\ldots,s_n)}(-a_0)^{n-(s_1+\cdots +s_n)}\,{s_1+\cdots+s_n\choose s_1,\ldots,s_n}\,a_1^{s_1}a_2^{s_2}\cdots a_{n}^{s_n},
\end{equation}
where the summation is over all integers $s_i\geq0$ satisfying
$s_1 +2s_2 +\cdots +ns_n = n$ and
$${s_1+\cdots+s_n\choose s_1,\ldots,s_n}=\frac{(s_1+\cdots+s_n)!}{s_1!\cdots s_n!}$$
denotes the multinomial coefficient.
\end{lem}
For example,
\begin{align*}
\det(M_4)&=(-a_0)^0\,{4\choose 4,0,0,0}\,a_1^4+(-a_0)^1\,{3\choose 2,1,0,0}\,a_1^2a_2+(-a_0)^2\,{2\choose 1,0,1,0}\,a_1a_3\\[0 pt]
&\quad+(-a_0)^2\,{2\choose 0,2,0,0}\,a_2^2+(-a_0)^3\,{1\choose 0,0,0,1}\,a_4\\[0 pt]
&=a_1^4-3a_0a_1^2a_2+2a_0^2a_1a_3+a_0^2a_2^2-a_0^3a_4.
\end{align*}
For brevity, we write $\det(\pm 1;a_1,a_2,\ldots,a_n)$ in place of $\det\left(M_n(\pm1;a_1,a_2,\ldots,a_n)\right)$.

\section{Tribonacci determinant formulas}

We find in this section determinant formulas for certain Toeplitz--Hessenberg matrices whose entries are various translates of the tribonacci sequence.  Our first result provides a connection between the tribonacci numbers and the Fibonacci and Padovan sequences.
\begin{thm}  \label{Theorem2} The following formulas hold:
\begin{align}
\label{tribFib}
\det(1;T_0,T_1,\ldots,T_{n-1})&=(-1)^{n-1}F_{n-2},\qquad n\geq 2,\\[0pt]
\label{tribPad}
\det(1;T_2,T_3,\ldots,T_{n+1})&=(-1)^{n-1}P_{n+2},\qquad n\geq 1.
\end{align}
\end{thm}
\begin{proof} To prove formula \eqref{tribFib}, we induct on $n$. The proof of (\ref{tribPad}) which we omit is similar.

Let $D_n=\det(1;T_0,T_1,\ldots,T_{n-1})$. The $n = 2$ and $n = 3$ cases of formula \eqref{tribFib} are easily verified. Suppose \eqref{tribFib} holds for all $k\leq n-1$, where $n\geq4$.  By recurrence \eqref{Det-Hess},  we have
\begin{align*}
D_{n}& = \sum_{i=1}^n(-1)^{i-1}T_{i-1}D_{n-i}\\[0pt]
&=T_0D_{n-1}-T_1D_{n-2}+T_2D_{n-3} +\sum_{i=4}^n(-1)^{i-1}\left(T_{i-2}+T_{i-3}+T_{i-4}\right)D_{n-i}
\end{align*}
\begin{align*}
&=D_{n-3}-\sum_{i=4}^n(-1)^{i}T_{i-2}D_{n-i}-\sum_{i=4}^n(-1)^{i}T_{i-3}D_{n-i}-\sum_{i=4}^n(-1)^{i}T_{i-4}D_{n-i}\\[0pt]
&=D_{n-3}+\!\sum_{i=3}^{n-1}(-1)^{i}T_{i-1}D_{n-i-1}-\!\sum_{i=2}^{n-2}(-1)^{i}T_{i-1}D_{n-i-2}+\!\sum_{i=1}^{n-3}(-1)^{i}T_{i-1}D_{n-i-3}\\[0pt]
&=D_{n-3}+\sum_{i=1}^{n-1}(-1)^{i}T_{i-1}D_{n-i-1}-\sum_{i=1}^{n-2}(-1)^{i}T_{i-1}D_{n-i-2}-D_{n-3}=-D_{n-1}+D_{n-2}\\[0pt]
&=-(-1)^{n-2}F_{n-3}+(-1)^{n-1}F_{n-4}=(-1)^{n-1}F_{n-2}.
\end{align*}

Consequently, formula \eqref{tribFib} is true in the $n$ case. Therefore, by induction, the formula holds for all $n\geq 2$.
\end{proof}
	
The next theorem provides the value of $\det(\pm1;a_1,a_2,\ldots, a_n)$ for some special tribonacci entries $a_i$.
\begin{thm} \label{Theorem3} Let $n\geq1$, except where stated otherwise. Then
\begin{align}
\det(-1;T_0,T_1,\ldots,T_{n-1})&=\left\lfloor\frac{2^{n}+6}{14}\right\rfloor,\quad\nonumber\\
\det(-1;T_0,T_2,\ldots,T_{2n-2})&=\frac{17+\sqrt{17}}{34}\left(\frac32+\frac{\sqrt{17}}{2}\right)^{n-2}+\frac{17-\sqrt{17}}{34}\left(\frac32-\frac{\sqrt{17}}{2}\right)^{n-2},\quad n\geq2,\nonumber\\
\det(1;T_1,T_2,\ldots,T_n)&=(-1)^{n-1}\sum_{i=0}^{\left\lfloor (n-2)/3\right\rfloor}{n-2-2i\choose i}\,\label{Det1all}\\
\det(-1;T_1,T_2,\ldots,T_n)&=\sum_{i=0}^{\left\lfloor(2n-4)/3\right\rfloor}{2n-4-2i\choose i}\,,\qquad n\geq2,\nonumber\\
\det(1;T_1,T_3,\ldots,T_{2n-1})&=(-1)^{n-1}\left\lfloor 4\cdot 3^{n-3}\right\rfloor,\nonumber\\
\det(1;T_3,T_4,\ldots,T_{n+2})&=0,\qquad n\geq4\,\label{T3T4..}\\
\det(1;T_3,T_5,\ldots,T_{2n+1})&=(-2)^{n-1}\sum_{i=0}^{n-1}2^{-i-\left\lfloor {i}/{2}\right\rfloor}\,{n-1-i\choose \left\lfloor{i}/{2}\right\rfloor}\,,\nonumber\\
\det(1;T_4,T_5,\ldots,T_{n+3})&= \begin{cases}
{\displaystyle (-1)^{n}}, &\text{\emph{if $n\equiv 0$}}\text{ \emph{(mod
3)}};\\
{\displaystyle (-1)^{n+1}}, &\text{\emph{if $n\equiv 1$}}\text{ \emph{(mod
3)}};\\
{\displaystyle 0}, &\text{\emph{if $n\equiv 2$}} \text{\emph{ (mod 3)}},\end{cases}\qquad n\geq2,\nonumber\\
\det(1;T_4,T_6,\ldots,T_{2n+2})&=4\cdot(-1)^{n-1},\qquad n\geq3,\nonumber\\
\det(1;T_5,T_6,\ldots,T_{n+4})&=\sum_{i=0}^{\left\lfloor(n+1)/2\right\rfloor}{n+2+i\choose n+1-2i}\,,\nonumber\\
\det(1;T_5,T_7,\ldots,T_{2n+3})&=4,\qquad n\geq3.\nonumber
\end{align}
\end{thm}
\begin{proof}
To prove formula (\ref{Det1all}), we induct on $n$; the others may also be shown inductively. Let $D_n = \det(1;T_1,T_2,\ldots,T_n).$
When $1 \leq n \leq 3$, the formula is seen to hold. Assume that \eqref{Det1all} is true for all $k \leq n-1$ where $n\geq 4$, and we prove it in the $n$ case.  By (\ref{Det-Hess}), we have
\begin{align*}
D_{n}& = \sum_{i=1}^n(-1)^{i+1}T_iD_{n-i}\nonumber\\
&=T_1D_{n-1}-T_2D_{n-2} +\sum_{i=3}^n(-1)^{i+1}\left(T_{i-1}+T_{i-2}+T_{i-3}\right)D_{n-i}\nonumber\\
&=-D_{n-2} -\sum_{i=3}^n(-1)^{i}T_{i-1}D_{n-i}-\sum_{i=3}^n(-1)^{i}T_{i-2}D_{n-i}-\sum_{i=3}^n(-1)^{i}T_{i-3}D_{n-i}\nonumber\\
&=-D_{n-2}+\sum_{i=1}^{n-1}(-1)^{i}T_{i}D_{n-i-1}-\sum_{i=1}^{n-2}(-1)^{i}T_{i}D_{n-i-2}+\sum_{i=1}^{n-3}(-1)^{i}T_{i}D_{n-i-3}\nonumber\\
&=-D_{n-2}-D_{n-1}+D_{n-2}-D_{n-3}=-D_{n-1}-D_{n-3}\nonumber\\
	&=(-1)^{n-1}\sum_{i=0}^{\left\lfloor\frac{n-3}{3}\right\rfloor}{n-3-2i\choose i} + (-1)^{n-1}\sum_{i=0}^{\left\lfloor\frac{n-5}{3}\right\rfloor}{n-5-2i \choose i}.\nonumber
\end{align*}
Thus,
\begin{equation}\label{theo1}
D_n=(-1)^{n-1}\!\left(\sum_{i=0}^{\left\lfloor\frac{n-3}{3}\right\rfloor}{n-3-2i\choose i} + \sum_{i=1}^{\left\lfloor\frac{n-2}{3}\right\rfloor}{n-3-2i \choose i-1}\right).
\end{equation}

Let $n\ne 3\ell-1$ for some $\ell$. Then $\left\lfloor\frac{n-3}{3}\right\rfloor=\left\lfloor\frac{n-2}{3}\right\rfloor$. From (\ref{theo1}), using the well-known recurrence
\begin{equation}\label{binomial}
{n-1 \choose k}+{n-1\choose k-1}={n\choose k},
\end{equation}
we then have
	\begin{align}
D_n&=(-1)^{n-1}\left(\sum_{i=0}^{\left\lfloor\frac{n-2}{3}\right\rfloor}{n-3-2i\choose i} + \sum_{i=0}^{\left\lfloor\frac{n-2}{3}\right\rfloor}{n-3-2i \choose i-1}\right)\nonumber\\
&=(-1)^{n-1}\sum_{i=0}^{\left\lfloor\frac{n-2}{3}\right\rfloor}{n-2-2i\choose i}.\nonumber
	\end{align}	
If $n=3\ell-1$, then
$\left\lfloor\frac{n-3}{3}\right\rfloor=\left\lfloor\frac{n-2}{3}\right\rfloor-1$. In this case, using (\ref{binomial}), we have
	\begin{align}
	D_n&=(-1)^{n-1}\left(\sum_{i=0}^{\left\lfloor\frac{n-2}{3}\right\rfloor-1}{n-3-2i\choose i} + \sum_{i=1}^{\left\lfloor\frac{n-2}{3}\right\rfloor}{n-3-2i \choose i-1}\right)\nonumber\\
	&=(-1)^{n-1}\left(\sum_{i=0}^{\left\lfloor\frac{n-2}{3}\right\rfloor}{n-3-2i\choose i}+ \sum_{i=0}^{\left\lfloor\frac{n-2}{3}\right\rfloor}{n-3-2i \choose i-1}\right)\nonumber\\
	&=(-1)^{n-1}\sum_{i=0}^{\left\lfloor\frac{n-2}{3}\right\rfloor}{n-2-2i\choose i}.\nonumber
	\end{align}	
Formula \eqref{Det1all} now follows by induction on $n$.
\end{proof}

\section{Multinomial versions of Toeplitz--Hessenberg determinant formulas}

Theorems \ref{Theorem2} and \ref{Theorem3} may be rewritten in terms of Trudi's formula \eqref{Trudi} as follows.
\begin{thm} \label{Theorem4} Let $n\geq1$, except where stated otherwise. Then
\begin{align}
	\sum_{(s_1,\ldots,s_n)}(-1)^{\sigma_n}p_n(s)T_0^{s_1}T_1^{s_2}\cdots T_{n-1}^{s_n}&=-F_{n-2},\qquad n \geq 2,\nonumber\\[0pt]
	\sum_{(s_1,\ldots,s_n)}p_n(s)T_0^{s_1}T_1^{s_2}\cdots T_{n-1}^{s_n}&=\left\lfloor\frac{2^n+6}{14}\right\rfloor,\nonumber\\[0pt]
\sum_{(s_1,\ldots,s_n)}p_n(s)T_0^{s_1}T_2^{s_2}\cdots T_{2n-2}^{s_n}&=\frac{17+\sqrt{17}}{34}\left(\frac32+\frac{\sqrt{17}}{2}\right)^{n-2}\nonumber\\[0pt]
&~~~~+\frac{17-\sqrt{17}}{34}\left(\frac32-\frac{\sqrt{17}}{2}\right)^{n-2},\qquad n\geq2,\label{formula}\\[0pt]
	\sum_{(s_1,\ldots,s_n)}(-1)^{\sigma_n}p_n(s)T_1^{s_1}T_2^{s_2}\cdots T_n^{s_n}&=-
	\sum_{i=0}^{\left\lfloor(n-2)/3\right\rfloor}{n-2-2i\choose i}\,,\nonumber\\[0pt]
\sum_{(s_1,\ldots,s_n)}p_n(s)T_1^{s_1}T_2^{s_2}\cdots T_n^{s_n}& =\sum_{i=0}^{\left\lfloor(2n-4)/3\right\rfloor}{2n-4-2i\choose i}\,,\qquad n\geq2,\nonumber\\[0pt]
	\sum_{(s_1,\ldots,s_n)}(-1)^{\sigma_n}p_n(s)T_1^{s_1}T_3^{s_2}\cdots T_{2n-1}^{s_n}&=-\left\lfloor 4\cdot 3^{n-3}\right\rfloor,\nonumber\\[0pt]
\sum_{(s_1,\ldots,s_n)}(-1)^{\sigma_n}p_n(s)T_2^{s_1}T_3^{s_2}\cdots T_{n+1}^{s_n}&=-P_{n+2},\label{formula1}\\[0pt]
	\sum_{(s_1,\ldots,s_n)}(-1)^{\sigma_n}p_n(s)T_3^{s_1}T_4^{s_2}\cdots T_{n+2}^{s_n}&=0,\qquad n\geq4,\label{formula2}\\
		\sum_{(s_1,\ldots,s_n)}(-1)^{\sigma_n}p_n(s)T_3^{s_1}T_5^{s_2}\cdots T_{2n+1}^{s_n}&=
		-2^{n-1}\sum_{i=0}^{n-1}2^{-i-\left\lfloor {i}/{2}\right\rfloor}{n-1-i\choose \left\lfloor {i}/{2}\right\rfloor}\,, \nonumber\\
	\sum_{(s_1,\ldots,s_n)}(-1)^{\sigma_n}p_n(s)T_4^{s_1}T_5^{s_2}\cdots T_{n+3}^{s_n}&=
\begin{cases}
{\displaystyle 1}, &\text{\emph{if $n\equiv 0$}}\text{ \emph{(mod
3)}};\\
{\displaystyle -1}, &\text{\emph{if $n\equiv 1$}}\text{ \emph{(mod
3)}};\\
{\displaystyle 0}, &\text{\emph{if $n\equiv 2$}} \text{\emph{ (mod 3)}},\end{cases}\, \qquad n\geq2,\nonumber\\[0pt]
\sum_{(s_1,\ldots,s_n)}(-1)^{\sigma_n}p_n(s)T_4^{s_1}T_6^{s_2}\cdots T_{2n+2}^{s_n}&=
	-4,\qquad n\geq3,\nonumber\\[0pt]
	\sum_{(s_1,\ldots,s_n)}(-1)^{\sigma_n}p_n(s)T_5^{s_1}T_6^{s_2}\cdots T_{n+4}^{s_n}&=(-1)^n
	\sum_{i=0}^{\left\lfloor(n+1)/2\right\rfloor}{n+2+i\choose n+1-2i}\,,\nonumber\\[0pt]
\sum_{(s_1,\ldots,s_n)}(-1)^{\sigma_n}p_n(s)T_5^{s_1}T_7^{s_2}\cdots T_{2n+3}^{s_n}&=4\cdot(-1)^n,\qquad n\geq3,\nonumber
	\end{align}
where the summation is over all $n$-tuples of integers $s_i\geq0$ satisfying $s_1 +2s_2 +\cdots +ns_n = n$,
$p_n(s)={s_1+\cdots+s_n\choose s_1,\ldots,s_n}$, and
$\sigma_n=s_1+\cdots+s_n$.
\end{thm}
\begin{ex} It follows for example from formulas \eqref{formula}, \eqref{formula2}, and \eqref{formula1}, respectively,  that
\begin{align*}
T_2^3 +2T_2T_6 + T_4^2 + T_{10}& = 100,\\
T_3^4 -3T_3^2T_4 + 2T_3T_5 + T_4^2 -T_6& = 0,\\
T_2^5-4T_2^3T_3+3T_2^2T_4+3T_2T_3^2-2T_2T_5-2T_3T_4+T_6 = P_7&=1.
\end{align*}
\end{ex}

\section{Generalized identities}

In this section, we will generalize the foregoing identities by combinatorial arguments. See, for example, \cite{GSh} or \cite{GSh2} for comparable combinatorial proofs involving determinants of matrices with Catalan or Fibonacci number entries, respectively.   In our arguments, we will make frequent use of the definition of the determinant of an $n\times n$ matrix $A=(a_{i,j})$ given by
$$\det(A)=\sum_{\sigma \in \mathcal{S}_n}(-1)^{\text{sgn}(\sigma)}a_{1,\sigma(1)}a_{2,\sigma(2)}\cdots a_{n,\sigma(n)},$$
where $\text{sgn}(\sigma)$ denotes the sign of a permutation $\sigma$.
Assume that permutations are expressed such that within each cycle, the first element is the smallest, where cycles are arranged from left to right in ascending order of smallest elements.  Note that if $A$ is Toeplitz-Hessenberg, then only permutations $\sigma$ where the elements within each cycle are increasing and comprise an interval contribute to the determinant sum above.  Upon regarding the various cycle lengths as parts, such $\sigma$ are synonymous with compositions $\rho$ of $n$, where $n$ is the size of $A$.  Assume that the sign of $\rho$ is the same as that of the associated $\sigma$; i.e., let $\rho$ have sign $(-1)^{n-\nu(\rho)}$, where $\nu(\rho)$ denotes the number of parts of $\rho$.

Let each part $i$ of $\rho$ be assigned the weight $a_i$ for all $i \geq 1$.  Then, for a Toeplitz-Hessenberg matrix $A$ of size $n$ with superdiagonal entry $a_0=1$, one may view $\det(A)$ as a weighted (signed) sum over the set of compositions $\rho$ of $n$, where the weight of $\rho$  is defined to be the product of the weights of its parts and the sign is as stated above. That is, if $\rho=(x_1,\ldots,x_m)$ where $x_1+\cdots+x_m=n$ with each $x_i \geq 1$, then the sign of $\rho$ is $(-1)^{n-m}$ and the weight is $\prod_{i=1}^m a_{x_i}$. On the other hand, if $a_0=-1$, then each term in the expansion of $\det(A)$ is positive, as the sign of the associated composition $\rho$ is always the same as the sign derived from the product of the superdiagonal elements in this case.  Thus, $\det(-1;a_1,\ldots,a_n)$ may be viewed as a sum of positively weighted compositions, with the weight of each individual composition being the same as before.

Note that compositions of $n$ are synonymous with linear \emph{tilings} of length $n$ where parts are identified as tiles of various lengths (tiles of the same length are understood to be indistinguishable).  The tilings themselves may be regarded as coverings of the members of $[n]$, written consecutively in a row (see, e.g., \cite[Chapter~1]{BQ}).  Here, we will be considering various kinds of restricted tilings where only pieces of a certain length may be used. Given $n \geq 1$ and $r \geq 3$, let $\mathcal{T}_n^{(r)}$ denote the set of tilings of length $n$ that use only pieces of length $1$, $2$, or $r$, which will be denoted respectively by $s$ for square, $d$ for domino, and $r$ for $r$-mino (where an $r$-mino is a $1\times r$ piece capable of covering $r$ consecutive numbers).  The set $\mathcal{T}_0^{(r)}$ when $n=0$ is understood to consist of the empty tiling having length zero.

Members of $\mathcal{T}_n^{(r)}$ will be referred to as \emph{generalized $r$-tribonacci tilings}. Note that when $r=3$, one gets the usual tribonacci tilings (see \cite[Section~3.3]{BQ}) whose pieces are $s$, $d$, and $t$ ($t$ standing for tromino).  Tilings of length $n$ that use only $s$ and $d$ are often referred to as \emph{square-and-domino} tilings and have cardinality given by the Fibonacci number $F_{n+1}$.  In the combinatorial proofs below, we will impose various restrictions on the positions of the different types of tiles and/or permit certain tiles to be marked or colored.  This allows for generalizations of the tribonacci determinant identities above in terms of the parameter $r$.

Let $T_n^{(r)}$ for $r \geq 3$ be defined by the recurrence $T_n^{(r)}=T_{n-1}^{(r)}+T_{n-2}^{(r)}+T_{n-r}^{(r)}$ for $n \geq r$, with $T_0^{(r)}=\cdots=T_{r-2}^{(r)}=0$ and $T_{r-1}^{(r)}=1$.  Note that $|\mathcal{T}_{n}^{(r)}|=T_{n+r-1}^{(r)}$ for $n \geq 0$ and that $T_n^{(r)}$ reduces to $T_n$ when $r=3$.  One could then refer to the $T_n^{(r)}$ as \emph{generalized $r$-tribonacci numbers}.  Let $P_n^{(r)}$ for $r \geq 3$ satisfy $P_n^{(r)}=P_{n-2}^{(r)}+P_{n-r}^{(r)}$ if $n \geq r$, with $P_0^{(r)}=1$ and $P_1^{(r)}=\cdots=P_{r-1}^{(r)}=0$.  The $P_n^{(r)}$ coincide with the  Padovan numbers when $r=3$.  Let $a_n^{(r)}$ be defined recursively by $a_n^{(r)}=a_{n-1}^{(r)}+a_{n-r}^{(r)}$ for $n \geq r$, with $a_0^{(r)}=\cdots=a_{r-1}^{(r)}=1$.

The following identities for determinants involving generalized $r$-tribonacci numbers reduce respectively when $r=3$ to formulas \eqref{tribFib}, \eqref{Det1all}, \eqref{tribPad}, and \eqref{T3T4..} above.
\begin{thm}\label{gth1}
If $r \geq 3$, then
\begin{equation}\label{gth1e1}
\det(1;T_0^{(r)},T_1^{(r)},\ldots,T_{n-1}^{(r)})=(-1)^{n-1}F_{n-r+1}, \qquad n \geq r-1,
\end{equation}
\begin{equation}\label{gth1e2}
\det(1;T_{r-2}^{(r)},T_{r-1}^{(r)},\ldots,T_{n+r-3}^{(r)})=(-1)^{n-1}a_{n-2}^{(r)}, \qquad n \geq 2,
\end{equation}
\begin{equation}\label{gth1e3}
\det(1;T_{r-1}^{(r)},T_{r}^{(r)},\ldots,T_{n+r-2}^{(r)})=(-1)^{n-1}P_{n+r-1}^{(r)}, \qquad n \geq 1,
\end{equation}
\begin{equation}\label{gth1e4}
\det(1;T_{r}^{(r)},T_{r+1}^{(r)},\ldots,T_{n+r-1}^{(r)})=(-1)^{n-1}\delta_{n,r}, \qquad n \geq 3,
\end{equation}
where $\delta_{n,r}$ is the Kronecker delta.
\end{thm}
\begin{proof}
Since \eqref{gth1e1} clearly holds for $n=r-1$, we may assume $n \geq r$.  Given $1 \leq k \leq n$, let $\mathcal{A}_n^{(k)}$ denote the set of all sequences of tilings $(\lambda_1,\ldots,\lambda_k)$ such that $\lambda_i\in\mathcal{T}_{\ell_i}^{(r)}$ where $\ell_1+\cdots+\ell_k=n-rk$ and $\ell_i \geq 0$ for all $i$.  Define the sign of a member of $\mathcal{A}_{n}^{(k)}$ to be $(-1)^{n-k}$ and let $\mathcal{A}_n=\cup_{k=1}^n\mathcal{A}_n^{(k)}$.  Since $T_0^{(r)}=\cdots=T_{r-2}^{(r)}=0$, all cycles in any permutation contributing a nonzero term in the expansion of $\det(1;T_0^{(r)},\ldots,T_{n-1}^{(r)})$ must have length at least $r$.  It is then seen that this expansion equals the sum of the signs of all members of $\mathcal{A}_n$.

For each $\lambda=(\lambda_1,\ldots,\lambda_k)\in \mathcal{A}_n^{(k)}$, we concatenate the $\lambda_i$ tilings (to form a single tiling of length $n-rk$), insert an $r$-mino between $\lambda_i$ and $\lambda_{i+1}$ for $i<k$ as well as after $\lambda_k$ and finally mark (indicated by underlining) each inserted $r$-mino to obtain
$\lambda'=\lambda_1\underline{r}\lambda_2\underline{r}\cdots \lambda_k\underline{r}$.  Doing so for all $\lambda \in \mathcal{A}_n$ implies that members of $\mathcal{A}_n$ may be identified as $r$-tribonacci tilings of length $n$ ending in an $r$-mino, where $r$-minos may be marked and the final $r$-mino is always marked, the set of which will be denoted by $\mathcal{A}_n'$.  Then $\det(1;T_0^{(r)},\ldots,T_{n-1}^{(r)})$ is seen to give the sum of the signs of all members of $\mathcal{A}_n'$, where the sign is given by  $(-1)^{n-(\#~\text{marked $r$-minos})}$.  We define a sign-changing involution of $\mathcal{A}_n'$ by identifying the rightmost non-terminal $r$-mino and either marking it or unmarking it, whichever applies.  This operation is defined on all members of $\mathcal{A}_n'$ except those containing a single $r$-mino, i.e., the terminal one.  Such members of $\mathcal{A}_n'$ then have sign $(-1)^{n-1}$ and their number is $F_{n-r+1}$ as they correspond to square-and-domino tilings of length $n-r$,  which implies formula \eqref{gth1e1}.

An analogous interpretation may be used to prove the determinant formula \eqref{gth1e2}. Here, let $\mathcal{B}_n$ denote the set of marked $r$-mino tilings of length $n$ in which dominos may be marked and whose final piece is a marked domino.  By similar reasoning as before, one has that $\det(1;T_{r-2}^{(r)},\ldots,T_{n+r-3}^{(r)})$ gives the sum of the signs of all members of $\mathcal{B}_n$, where the sign equals $(-1)^{n-(\#~\text{marked dominos})}$.  Define an involution on $\mathcal{B}_n$ by switching the rightmost non-terminal domino to the other option, if possible.  The survivors of this involution each have sign $(-1)^{n-1}$ and are synonymous with tilings of length $n-2$ containing only square and $r$-mino pieces.  Thus, they have cardinality $a_{n-2}^{(r)}=\sum_{i=0}^{\lfloor\frac{n-2}{r}\rfloor}\binom{n-2-(r-1)i}{i}$.

Similar reasoning applies to \eqref{gth1e3} and \eqref{gth1e4}, where instead of $\mathcal{B}_n$, one considers sets $\mathcal{C}_n$ and $\mathcal{D}_n$ that are defined the same way as $\mathcal{B}_n$ except that instead of marking dominos, one marks squares for $\mathcal{C}_n$ and any of the three kinds of pieces for $\mathcal{D}_n$.  One then defines the sign and the involution comparably as before in either case.  Note that in \eqref{gth1e3}, the $r$-Padovan number $P_{n+r-1}^{(r)}$ enumerates the survivors of the involution as they are tilings of length $n-1$ that use only dominos and $r$-minos. For \eqref{gth1e4}, note that if $n \geq r+1$ or $3 \leq n \leq r-1$, then each member of $\mathcal{D}_n$ must contain at least two pieces (and, in particular, a non-terminal piece).  Hence, there are no survivors of the involution in these cases, implying that the determinant is zero.  If $n=r$, then the sole survivor is the tiling consisting of a single (marked) $r$-mino, implying the determinant is $(-1)^{r-1}$ in this case.
\end{proof}

The determinant formulas in the following theorem involve those in which $a_0=1$ and $a_1=T_{r+1}^{(r)}$, with the subsequent $a_i$ subscripts increasing by either one or two. They extend the comparable identities above where $a_1=T_4$.
\begin{thm}\label{gth2}
If $r \geq 3$, then
\begin{equation}\label{gth2e1}
\det(1;T_{r+1}^{(r)},T_{r+2}^{(r)},\ldots,T_{n+r}^{(r)})=\sum_{i=0}^{\left\lfloor\frac{n}{r-1}\right\rfloor}(-1)^{ri}\binom{n-(r-2)i}{i}, \qquad n \geq 2,
\end{equation}
and \small
\begin{equation}\label{gth2e2}
\det(1;T_{r+1}^{(r)},T_{r+3}^{(r)},\ldots,T_{2n+r-1}^{(r)})=\begin{cases}
{\displaystyle 4\cdot(-1)^{n-1}, \quad n \geq r}, &\emph{if}~r~\emph{is~odd};\\
{\displaystyle (-1)^{n-1}\sum_{i=0}^{\left\lfloor\frac{2(n-1)}{r}\right\rfloor}\binom{n-1-\left(\frac{r}{2}-1\right)i}{i}, \quad n \geq 1}, &\emph{if}~r~\emph{is~even}.\end{cases}
\end{equation}\normalsize
\end{thm}
\begin{proof}
Given $1 \leq k \leq n$, let $\mathcal{E}_n^{(k)}$ denote the set of sequences of tilings $\lambda=(\lambda_1,\ldots,\lambda_k)$ such that $\lambda_j$ is an $r$-tribonacci tiling of length $\ell_j$, where $\ell_1+\cdots+\ell_k=n+k$ and $\ell _j \geq 2$ for all $j$, and let $\mathcal{E}_n=\cup_{k=1}^n\mathcal{E}_n^{(k)}$.  Define the sign of $\lambda \in \mathcal{E}_n^{(k)}$ by $(-1)^{n-k}$.  Then the left side of \eqref{gth2e1} gives the sum of the signs of all members of $\mathcal{E}_n$.  Let $\mathcal{E}_n^* \subseteq \mathcal{E}_n$ consist of those sequences $\lambda=(\lambda_1,\lambda_2,\ldots)$ such that $\lambda_j=d\text{ or } r$ for all $j$ (that is, each $\lambda_j$ is the tiling consisting of a single domino or $r$-mino piece).  Note that the sum of the signs of all members of $\mathcal{E}_n^*$ is given by the right-hand side of \eqref{gth2e1}, upon considering the number of times $r$ is chosen.  For if $r$ is chosen exactly $i$ times as a tiling within $\lambda$, then we must have $n-(r-1)i$ $d$'s in $\lambda$ and hence $k=n-(r-2)i$. This implies that there are $\binom{n-(r-2)i}{i}$ such members of $\mathcal{E}_n^*$, each of sign $(-1)^{n-k}=(-1)^{(r-2)i}=(-1)^{ri}$.

To complete the proof of \eqref{gth2e1}, we define a sign-changing involution of $\mathcal{E}_n-\mathcal{E}_n^*$.  Consider first the pairing
$$\lambda=(\lambda_1,\ldots,\lambda_k=\alpha s) \leftrightarrow \lambda'=(\lambda_1,\ldots,\lambda_{k-1},\lambda_k'=\alpha,\lambda_{k+1}'=ss),$$
where $|\alpha|\geq 2$, if the final tiling ends in $s$.  So assume that the final tiling of $\lambda=(\lambda_1,\lambda_2,\ldots)\in \mathcal{E}_n-\mathcal{E}_n^*$ does not end in $s$. Let $j_o$ be the largest index $j$ such that $\lambda_j$ is either of the form (i) $\lambda_j=\beta d \text{ or } \beta r$ or (ii) $\lambda_j=\beta s$, where $|\beta|\geq 1$.  Note that by the assumptions on $\lambda$, such an index $j_o$ always exists and that if (ii) holds, then $\lambda_{j_o}$ is not the final tiling.  Consider the replacement of $\lambda_{j_o}=\beta d \text{ or } \beta r$ with $\lambda_{j_o}=\beta s$, $\lambda_{j_o+1}=d \text{ or } r$ if (i) holds, and vice versa, if (ii).  This operation yields the desired involution of $\mathcal{E}_n-\mathcal{E}_n^*$, as one may verify.

For \eqref{gth2e2}, first let $\mathcal{F}_n$ denote the set of marked $r$-tribonacci tilings of length $2n$ wherein any tile (including a square) whose rightmost section corresponds to an even-numbered position may be marked, with the final tile always marked.  Define the sign as $(-1)^{n-(\text{$\#$ marked tiles})}$.  By similar reasoning as before, the determinant in \eqref{gth2e2} gives the sum of the signs of all members of $\mathcal{F}_n$.  We define an involution on $\mathcal{F}_n$ by identifying the rightmost non-terminal tile that ends on an even-numbered position and either marking or unmarking it.  If $r$ is odd, then the set $S$ of survivors for $n \geq r$ is given by
$$S=\{sd^{n-1}s,rd^{n-\left(\frac{r+1}{2}\right)}s,sd^{n-\left(\frac{r+1}{2}\right)}r,rd^{n-r}r\}.$$
Since only the terminal piece can be marked, each member of $S$ has sign $(-1)^{n-1}$, which implies the odd case of \eqref{gth2e2}.  If $r$ is even, then
$$S=\{s\alpha s: \alpha \in \mathcal{T}_{2n-2}^{(r)} \text{ and contains no squares}\}.$$
Considering the number $i$ of $r$-minos within $\alpha$, and then halving the size of each tile, implies the even case of \eqref{gth2e2} and completes the proof.
\end{proof}

\noindent \emph{Remark:}  The preceding involution shows further in the odd case of $r$ that
\begin{equation*}
\det(1;T_{r+1}^{(r)},T_{r+3}^{(r)},\ldots,T_{2n+r-1}^{(r)})=\begin{cases}
{\displaystyle 3\cdot(-1)^{n-1}}, &\emph{\emph{if}}~\frac{r+1}{2}\leq n<r;\\
{\displaystyle (-1)^{n-1}}, &\emph{\emph{if}}~2 \leq n<\frac{r+1}{2}.\end{cases}
\end{equation*}
Note also that the $r=3$ case of \eqref{gth2e1} reduces to the corresponding identity in Theorem \ref{Theorem3} above by the well-known formula \cite[Identity~175]{BQ}
$$\sum_{i=0}^{\lfloor{n}/{2}\rfloor}(-1)^{i}\binom{n-i}{i}=\begin{cases}
{\displaystyle (-1)^{n}}, &\text{\emph{\emph{if} $n\equiv 0$}}\text{ (mod
3)};\\
{\displaystyle (-1)^{n+1}}, &\text{\emph{\emph{if} $n\equiv 1$}}\text{ (mod
3)};\\
{\displaystyle 0}, &\text{\emph{\emph{if} $n\equiv 2$}} \text{ (mod 3)}.\end{cases}$$\\
\begin{thm}\label{gth3}
Let $r \geq 3$.  If $r$ is odd, then
\begin{equation}\label{gth3e1}
\det\big(1;T_1^{(r)},T_3^{(r)},\ldots,T_{2n-1}^{(r)}\big)=(-1)^{n-1}a_n,
\end{equation}
where $a_n=3a_{n-1}-a_{n-2}+a_{n-\frac{r+1}{2}}$, $n \geq r+1$, with $a_r=1+F_{r+1}$, $a_n=F_{2n-r+1}$ for $\frac{r+1}{2}\leq n <r$, and $a_n=0$ for $1 \leq n<\frac{r+1}{2}$.

If $r$ is even, then
\begin{equation}\label{gth3e2}
\det\big(1;T_1^{(r)},T_3^{(r)},\ldots,T_{2n-1}^{(r)}\big)=(-1)^{n-1}b_n,
\end{equation}
where $b_n=3b_{n-1}-b_{n-2}+b_{n-\frac{r}{2}}-b_{n-\frac{r}{2}-1}$, $n\geq r+1$, with $b_n=F_{2n-r+1}$ for $\frac{r}{2}\leq n\leq r$ and $b_n=0$ for $1 \leq n<\frac{r}{2}$.
\end{thm}
\begin{proof}
Let us refer to an $r$-mino within a tiling as \emph{even} or \emph{odd}, depending on whether its rightmost section corresponds to an even- or odd-numbered position, respectively.  Let $\mathcal{H}_n$ denote the set of $r$-tribonacci tilings of length $2n$ wherein even $r$-minos may be marked, with the final piece a marked $r$-mino.  Let $\lambda \in \mathcal{H}_n$ have sign $(-1)^{n-(\text{\# marked $r$-minos})}$.  Then it is seen that $\det(1;T_1^{(r)},\ldots,T_{2n-1}^{(r)})$ gives the sum of the signs of all $\lambda$.  Define a sign-changing involution on $\mathcal{H}_n$ by either marking or unmarking the rightmost non-terminal even $r$-mino.  This operation is not defined on the subset $\mathcal{H}_n^*$ of $\mathcal{H}_n$ whose members contain no even $r$-mino outside of the terminal $r$-mino.  Each member of $\mathcal{H}_n^*$ has sign $(-1)^{n-1}$, so we must determine $|\mathcal{H}_n^*|$.  Note that members of $\mathcal{H}_n^*$ are synonymous with tilings in $\mathcal{T}_{2n-r}^{(r)}$ containing only odd $r$-minos.

First assume $r$ is odd and let $a_n=|\mathcal{H}_n^*|$.  If $1 \leq n <\frac{r+1}{2}$, then $\mathcal{H}_n^*$ is empty, whereas if $\frac{r+1}{2} \leq n <r$, then members of $\mathcal{H}_n^*$ correspond to square-and-domino tilings and thus have cardinality $F_{2n-r+1}$.  If $n=r$, then the tiling consisting of a single $r$-mino is also possible since it would be odd in this case, which implies $a_r=1+F_{r+1}$.  If $n \geq r+1$, then
$\lambda \in \mathcal{H}_n^*$ ending in $ss$ or $d$ may be obtained from members of $\mathcal{H}_{n-1}^*$ by appending the respective suffix.  Furthermore, appending $sr$ to an arbitrary tiling in $\mathcal{H}_{n-\frac{r+1}{2}}^*$ gives those $\lambda$ with this suffix.  Members of $\mathcal{H}_n^*$ ending in $ds$ or $dr$ may be obtained from those in $\mathcal{H}_{n-1}^*$ not ending in $d$ by inserting a $d$ directly prior to the final piece.  Since $\lambda \in \mathcal{H}_n^*$ cannot end in $rr$ or $rs$, all members of $\mathcal{H}_n^*$ have thus been accounted for.  Combining the previous cases then implies $a_n=2a_{n-1}+a_{n-\frac{r+1}{2}}+(a_{n-1}-a_{n-2})$, which completes the odd case.

Now assume $r$ is even and let $b_n=|\mathcal{H}_n^*|$ in this case.  Upon verifying the initial conditions, one may assume $n \geq r+1$.  Then there are clearly $2a_{n-1}$ members of $\mathcal{H}_n^*$ ending in either $ss$ or $d$.  Since $r$ even implies tilings in $\mathcal{H}_n^*$ are of even length, the final piece cannot be $r$.  Thus, the remaining possibilities are tilings of the form $\lambda=\alpha s \beta ds$ or $\lambda=\alpha s\beta rs$, where $\alpha$ and $\beta$ are possibly empty and $\beta$ does not contain $s$.  Upon removal of the rightmost $d$ or $r$, it is apparent that there are $a_{n-1}-a_{n-2}$ or $a_{n-\frac{r}{2}}-a_{n-\frac{r}{2}-1}$ possibilities, respectively, for tilings of the stated form, by subtraction.  Combining the preceding cases gives \eqref{gth3e2} and completes the proof.
\end{proof}

Using the preceding result, one can show the generating function $f(x)=\sum_{n\geq 1}a_nx^n$ for the $r$ odd case is given by
$$f(x)=\frac{x^{\frac{r+1}{2}}+x^r}{1-3x+x^2-x^{\frac{r+1}{2}}},$$
whereas $g(x)=\sum_{n\geq1}b_nx^n$ in the even case is given by
$$g(x)=\frac{x^{\frac{r}{2}}(1-x-x^\frac{r}{2})}{1-3x+x^2-x^{\frac{r}{2}}+x^{\frac{r}{2}+1}}.$$

\noindent\emph{Remark:} Letting $r=3$ in \eqref{gth3e1} implies $a_n=3a_{n-1}$ for $n \geq 4$, with $a_3=4$, which yields the formula found in the previous section for $\det(1;T_1,T_3,\ldots,T_{2n-1})$.  Taking $r=4$ in \eqref{gth3e2} implies $b_n$ satisfies the third-order recurrence $b_n=3b_{n-1}-b_{n-3}$ for $n \geq 5$, with $b_2=1$, $b_3=2$, and $b_4=5$.\medskip

\begin{thm}\label{gth4}
Let $n \geq 1$ and $r \geq 3$.  Then
\begin{equation}\label{gth4e1}
\det(1;T_r^{(r)},T_{r+2}^{(r)},\ldots,T_{2n+r-2}^{(r)}) = \begin{cases}
{\displaystyle (-1)^{n-1}k_n^{(r)}}, &\emph{if}~r~\emph{is~odd};\\
{\displaystyle (-1)^{n-1}\sum_{i=0}^{n-1}a_i^{(r/2)}a_{n-1-i}^{(r/2)}}, &\emph{if}~r~\emph{is~even},\end{cases}
\end{equation}
where $k_n^{(r)}$ is determined by $$\sum_{n\geq1}k_n^{(r)}x^n=\frac{x+x^{\frac{r+1}{2}}}{1-2x+x^2-x^{\frac{r+1}{2}}-x^r}.$$
In particular, for $r=3$, we have $$\det(1;T_3,T_5,\ldots,T_{2n+1})=(-2)^{n-1}\sum_{i=0}^{n-1}2^{-i-\left\lfloor {i}/{2}\right\rfloor}\,{n-1-i\choose \left\lfloor{i}/{2}\right\rfloor}.$$
\end{thm}
\begin{proof}
Let $\mathcal{K}_n=\mathcal{K}_n^{(r)}$ denote the set of $r$-tribonacci tilings of length $2n$ in which squares occurring in even-numbered positions may be marked and whose final piece is a marked square.  Define the sign of $\lambda \in \mathcal{K}_n$ by $(-1)^{n-(\#~\text{marked~squares})}$. Note that $\lambda$ may be written as $\lambda=\lambda_1\lambda_2\cdots \lambda_\ell$ for some $\ell \geq 1$ such that $\lambda_i=\lambda_i's$ for $1 \leq i \leq \ell$, where the terminal $s$ is marked and $\lambda_i'$ contains no marked squares.  Note that if $\lambda_i'$ has length $k_i$ for each $i$, then there are $\prod_{i=1}^\ell T_{r+k_i-1}^{(r)}$ possibilities for the various $\lambda_i'$ once $\ell$ and the $k_i$ are specified.  Furthermore, we have $k_1+\cdots+k_\ell=2n-\ell$ with each $k_i$ odd.  Since $\ell$ may be identified as the number of cycles within a permutation that contributes a non-zero term in the expansion of $\det(1;T_r^{(r)},\ldots,T_{2n+r-2}^{(r)})$, it follows that the sum of the signs of all members of $\mathcal{K}_n$ is given by $\det(1;T_r^{(r)},\ldots,T_{2n+r-2}^{(r)})$.

We define a sign-changing involution of $\mathcal{K}_n$ by identifying the rightmost square that is in an even but not in the terminal position and either marking it if it is unmarked or removing the marking from it if marked.  This involution is not defined on the subset $\mathcal{K}_n^*$ of $\mathcal{K}_n$ consisting of those tilings in which squares (necessarily unmarked) occur only in odd-numbered positions.  Note that each member of $\mathcal{K}_n^*$ has sign $(-1)^{n-1}$ since only the terminal square is marked.  To determine $|\mathcal{K}_n^{*}|$, we consider cases based on the parity of $r$.  If $r$ is even, then members of $\mathcal{K}_n^*$ can contain only one square in addition to the terminal square and thus must be of the form $\lambda=\lambda's\lambda''s$, where the sections $\lambda'$ and $\lambda''$ contain no squares.  If $2i$ denotes the length of $\lambda'$, where $0 \leq i \leq n-1$, then considering all possible $i$ implies $|\mathcal{K}_n^{*}|=\sum_{i=0}^{n-1}a_i^{(r/2)}a_{n-1-i}^{(r/2)}$, upon halving the length of each piece within a tiling.

Assume now $r=3$ and let us find $|\mathcal{K}_n^{*}|$ in this case.  The proof of \eqref{gth4e1} for general $r$ odd, which we outline briefly below, will follow similarly.   We may clearly assume $n \geq 2$.  Let $\alpha=\alpha_0s\alpha_1\cdots s\alpha_ps \in \mathcal{K}_n^*$, where only the terminal $s$ is marked, $p \geq 0$, and each section $\alpha_0,\ldots,\alpha_p$ contains no squares.  Note that if $p\geq 1$, then $\alpha_0$ and $\alpha_p$ must be of even length, with $\alpha_1,\ldots,\alpha_{p-1}$ of odd length.  Let $\Delta_n$ denote the set of tilings of length $n$ using only square and tromino pieces such that the squares come in two colors.  Since
\begin{align*}
|\Delta_{n-1} \cup \Delta_{n-2}|&=\sum_{i=0}^{\left\lfloor\frac{n-1}{2}\right\rfloor}2^{n-1-3i}\binom{n-1-2i}{i}+\sum_{i=0}^{\left\lfloor\frac{n-2}{2}\right\rfloor}2^{n-2-3i}\binom{n-2-2i}{i}\\
&=\sum_{i=0}^{n-1}2^{n-1-i-\left\lfloor{i}/{2}\right\rfloor}\binom{n-1-i}{\left\lfloor{i}/{2}\right\rfloor},
\end{align*}
to establish the $r=3$ case, it suffices to define a bijection $h$ between $\mathcal{K}_n^*$ and $\Delta_{n-1} \cup \Delta_{n-2}$.

To do so, we will make use of the decomposition of $\alpha$ given above and consider each of the sections $\alpha_i$ separately, starting with $\alpha_0$ and first assuming $p\geq 1$.  Note that since $\alpha_0$ is of even length, it must contain an even number of trominos and thus we have
$\alpha_0=d^{i_1}td^{i_2}td^{i_3}\cdots td^{i_{2a+1}}$
for some $a \geq 0$, where $i_j \geq 0$ for all $1 \leq j \leq 2a+1$.  Define
$$(\alpha_0)'=s_2^{i_1}s_1^{i_2}ts_2^{i_3}s_1^{i_4}t\cdots s_2^{i_{2a-1}}s_1^{i_{2a}}ts_2^{i_{2a+1}},$$
where $s_1$ and $s_2$ denote squares of two different colors.  If $1 \leq k \leq p-1$, then the section $\alpha_k$ of $\alpha$ contains an odd number of trominos and thus is of the form
$\alpha_k=d^{j_1}td^{j_2}t\cdots d^{j_{2b-1}}td^{j_{2b}}$ for some $b \geq 1$, where $j_i \geq 0$ for $1 \leq i \leq 2b$.  For each $k$, let
$$(s\alpha_k)'=s_1^{j_1+1}s_2^{j_2+1}(s_1^{j_3}ts_2^{j_4})\cdots (s_1^{j_{2b-1}}ts_2^{j_{2b}}).$$
Next observe that the final section $\alpha_p$ must contain an even number of trominos (possibly none).  If $\alpha_p=d^c$ for some $c \geq 0$, then let $(s\alpha_p s)'=s_1^c$.  Otherwise, $\alpha_p$ is of the form $\alpha_p=d^{j_1}td^{j_2}t\cdots d^{j_{2b}}td^{j_{2b+1}}$ where $b \geq 1$, in which case we put
$$(s\alpha_ps)'=s_1^{j_1+1}s_2^{j_2+1}(s_1^{j_3}ts_2^{j_4})\cdots(s_1^{j_{2b-1}}ts_2^{j_{2b}})s_1^{j_{2b+1}}.$$
Let $h(\alpha)$ be given by
$h(\alpha)=(\alpha_0)'(s\alpha_1)'\cdots(s\alpha_{p-1})'(s\alpha_ps)'$, where the various tilings are concatenated.  Note that if $\alpha_p$ contains no trominos, then $h(\alpha) \in \Delta_{n-1}$, whereas if $\alpha_p$ contains a tromino (hence, at least two), then $h(\alpha) \in \Delta_{n-2}$.

If $p=0$, then $\alpha=\alpha_0s$ with $\alpha_0=d^{i_1}t\cdots d^{i_{2a-1}}td^{i_{2a}}$, where $a \geq 1$, and we set
$$h(\alpha)=s_2^{i_1}s_1^{i_2}t\cdots s_2^{i_{2a-3}}s_1^{i_{2a-2}}ts_2^{i_{2a-1}}s_1^{i_{2a}} \in \Delta_{n-2}.$$
Note that in this case we obtain all members of $\Delta_{n-2}$ that do not contain an $s_2$ directly following an $s_1$, which were missed in the preceding case.

To reverse $h$, first suppose $\delta\in \Delta_{n-2}$ and consider the number of times that $s_2$ directly follows $s_1$ (which will correspond to $p$ in the decomposition of $h^{-1}(\delta)$ in this case).  If this does not occur, then simply reverse the last operation above.  Otherwise, one may divide up $\delta$ accordingly into sections each starting with a non-empty run of $s_1$ directly followed by a non-empty run of $s_2$ (with $\alpha_0$ comprising any remaining initial tiles).  Then by reversing the prime operation above in each of the cases, one can reconstruct the sections $\alpha_0,\ldots,\alpha_p$, and hence $\alpha$.  If $\delta \in \Delta_{n-1}$, then proceed similarly except that we let $\alpha_p=d^c$, where $c$ denotes the length of a (possibly empty) terminal run of $s_1$.  This completes the proof of the $r=3$ case.

One may generalize the argument above when $r=3$ to the general $r$ odd case.  Doing so implies that the generating function for the number of members of $\mathcal{K}_n^*$ of the form $\alpha_0s\cdots\alpha_ps$ where $p \geq 0$ is given by
\begin{align*}
&x^{p\left(\frac{r-3}{2}\right)+1}\cdot x^{2p}\left(\frac{1}{1-2x+x^2-x^r}\right)^{p+1}+x^{(p-1)\left(\frac{r-3}{2}\right)+r-1}\cdot x^{2p}\left(\frac{1}{1-2x+x^2-x^r}\right)^{p+1}\notag\\
&=\left(x^{-\left(\frac{r-1}{2}\right)}+1\right)\left(\frac{x^{\frac{r+1}{2}}}{1-2x+x^2-x^r}\right)^{p+1}.
\end{align*}
Summing over all $p \geq 0$ implies
$$\sum_{n \geq 1}|\mathcal{K}_n^*|x^n=\frac{x+x^{\frac{r+1}{2}}}{1-2x+x^2-x^{\frac{r+1}{2}}-x^r},$$
from which the general odd case follows since $\det(1;T_r^{(r)},\ldots,T_{2n+r-2}^{(r)})=(-1)^{n-1}|\mathcal{K}_n^*|$.
\end{proof}

\noindent\emph{Remark:} Taking $r=3$ in \eqref{gth4e1} gives
$$\sum_{n\geq1}\det(1;T_3,T_5,\ldots,T_{2n+1})x^n=\frac{x-x^2}{1+2x+x^3},$$
which yields the explicit formula stated above in this case.  Note that one may find a formula for the coefficients $k_n$ in terms of a double sum of binomial coefficients, upon considering cases on the parity of $n$ and whether $r$ is congruent to $1$ or $3$ (mod $4$), the details of which we leave to the reader.\medskip
\begin{thm}\label{gth5}
Let $r \geq 3$ be odd and $n \geq \frac{r+3}{2}$.  If $r \geq 7$, then
\begin{equation}\label{gth5e1}
\det(1;T_{r+2}^{(r)},T_{r+4}^{(r)},\ldots,T_{2n+r}^{(r)}) = \begin{cases}
{\displaystyle (-1)^{n-q}}, &\emph{if}~{\displaystyle n=q(r-1)/2};\\
{\displaystyle 2\cdot(-1)^{n-1-q}}, &\emph{if}~{\displaystyle n-1=q(r-1)/2};\\
{\displaystyle (-1)^{n-q}}, &\emph{if}~{\displaystyle n-2=q(r-1)/2};\\
{\displaystyle 0}, &\emph{otherwise}.
\end{cases}
\end{equation}
If $r = 5$, then
\begin{equation}\label{gth5e2}
\det(1;T_7^{(5)},T_9^{(5)},\ldots,T_{2n+5}^{(5)})=\begin{cases}
{\displaystyle 0}, &\emph{if}~n\equiv 0 \text{ \emph{(mod 2)}};\\
{\displaystyle 2\cdot(-1)^{\frac{n-1}{2}}}, &\emph{if}~n\equiv 1 \text{ \emph{(mod 2)}}.
\end{cases}
\end{equation}
If $r=3$, then $\det(1;T_5,T_7,\ldots,T_{2n+3})=4$ for all $n \geq 3$.
\end{thm}
\begin{proof}
Given $n \geq \frac{r+3}{2}$ and $1 \leq a \leq n$, let $\mathcal{L}_n^{(a)}$ denote the set of sequences $\lambda=(\lambda_1,\ldots,\lambda_a)$ of tilings such that $\lambda_i \in \mathcal{T}_{2\ell_i+1}^{(r)}$ for $1 \leq i \leq a$ where $\ell_1+\cdots+\ell_a=n$ and $\ell_i \geq 1$ for all $i$. Let $\lambda \in \mathcal{L}_n^{(a)}$ have sign $(-1)^{n-a}$ and $\mathcal{L}_n=\cup_{a=1}^{n}\mathcal{L}_n^{(a)}$.  From the definition of the determinant, the sum of the signs of the members of $\mathcal{L}_n$ is seen to be given by $\det(1;T_{r+2}^{(r)},\ldots,T_{2n+r}^{(r)})$.  We define preliminarily an involution on $\mathcal{L}_n$ as follows.  Given $\lambda \in \mathcal{L}_n^{(a)}$ where $1 \leq a<n$, let $\lambda'=(\lambda_1,\ldots,\lambda_{a-1},\lambda_a',\lambda_{a+1}')$ be obtained from $\lambda$ by making the indicated replacement of the tiling $\lambda_a$ in each case:
\begin{itemize}
\item $\lambda_a=\alpha d$, $|\alpha|\geq 3$ $\rightarrow$ $\lambda_a'=\alpha,~\lambda_{a+1}'=sd$,
\item $\lambda_a=\alpha ss$, $|\alpha|\geq 3$  $\rightarrow$ $\lambda_a'=\alpha,~\lambda_{a+1}'=sss$,
\item $\lambda_a=\alpha rs$, $|\alpha|\geq 3$ $\rightarrow$ $\lambda_a'=\alpha,~\lambda_{a+1}'=srs$,
\item $\lambda_a=\alpha ds$, $|\alpha|\geq 2$ $\rightarrow$ $\lambda_a'=\alpha s,~\lambda_{a+1}'=ds$,
\item $\lambda_a=\alpha r$, $|\alpha|\geq 2$  $\rightarrow$ $\lambda_a'=\alpha s,~\lambda_{a+1}'=r$.
\end{itemize}
One may verify in each case that $\lambda'$ indeed belongs to $\mathcal{L}_n^{(a+1)}$.

Let $U=U_n$ be given by $U:=\{(r,\ldots,r),(r,\ldots,r,ds),(sd,r,\ldots,r),(sd,r,\ldots,r,ds)\}\cap\mathcal{L}_n$.  Note that $U$ may consist of up to four elements or be empty depending on $n$.  We extend the prime operation defined above to $\mathcal{L}_n-U$ by identifying the rightmost tiling $\rho_i$ within $\rho=(\rho_1,\rho_2,\ldots)\in \mathcal{L}_n-U$ such that $\rho_i$ (or possibly $\rho_i$, together with $\rho_{i-1}$) disqualifies $\rho$ from membership in $U$ and applying the appropriate operation (or its inverse) from those defined above to $\rho_i$ (or to $\rho_i$ and $\rho_{i-1}$) so as to obtain $\rho'$.  For example, if $n=10$, $r=5$, and $\rho=(sd,s^2ds,d^2s,r,r,ds)\in \mathcal{L}_{10}^{(6)}$, then $\rho_4=r$, taken together with $\rho_3=d^2s$, is the rightmost tiling that disqualifies $\rho$ from belonging to $U$, which implies $\rho'=(sd,s^2ds,d^2r,r,ds)\in \mathcal{L}_{10}^{(5)}$.  One may verify that the operation $\rho \mapsto \rho'$ defines a sign-changing involution of $\mathcal{L}_n-U$ in all cases.

Thus, the determinant in question equals the sum of the signs of the members of $U$.  If $r \geq 7$, then $U$ is empty if $n \not\equiv 0,1,2 \text{ mod}\left(\frac{r-1}{2}\right)$, whence the determinant is zero in these cases.  Otherwise, note that each tiling consisting of a single $r$ contributes $(-1)^{(r-3)/2}$ towards the sign.  Thus, if $n$ is divisible by $\frac{r-1}{2}$, then $U$ consists of a single element $(r,\ldots,r)$ whose sign is $(-1)^{q(r-3)/2}=(-1)^{n-q}$, where $n=q(r-1)/2$.  If $n \equiv 1 \text{ mod}\left(\frac{r-1}{2}\right)$, then $U$
contains two elements each having sign $(-1)^{n-1-q}$, whereas if $n \equiv 2 \text{ mod}\left(\frac{r-1}{2}\right)$, then $U$ consists of a single element of sign $(-1)^{n-q}$, where $q$ is as given above in the respective cases.  If $r=5$ and $n \geq 4$, then $U$ contains two elements of opposite sign, if $n$ is even, and of the same sign, if $n$ is odd.  If $r=3$, then $U$ contains four members each belonging to $\mathcal{L}_n^{(n)}$ for all $n \geq 3$, which implies the last statement. \end{proof}

Direct computations yield the following further expressions for the generating function of certain $r$-tribonacci determinants.

\begin{thm}\label{gth6}
If $r\geq 3$, then
\begin{equation}\label{gth6e1}
\sum_{n\geq1}\det(1;T_{0}^{(r)},T_{2}^{(r)},\ldots,T_{2n-2}^{(r)})x^n=\begin{cases}
{\displaystyle \frac{(-x)^{\frac{r+1}{2}}(-1-x)}{1+3x+x^2-(-x)^{\frac{r+1}{2}}-(-x)^{\frac{r+3}{2}}+x^r}}, &\emph{if}~r~\emph{is~odd};\\
{\displaystyle \frac{-(-x)^{\frac{r}{2}+1}}{1+3x+x^2-2(-x)^{\frac{r}{2}}+3(-x)^{\frac{r}{2}+1}+x^r}}, &\emph{if}~r~\emph{is~even},\end{cases}
\end{equation}

\begin{equation}\label{gth6e2}
\sum_{n\geq1}\det(-1;T_{0}^{(r)},T_{2}^{(r)},\ldots,T_{2n-2}^{(r)})x^n=\begin{cases}
{\displaystyle \frac{x^{\frac{r+1}{2}}(1-x)}{1-3x+x^2-3x^{\frac{r+1}{2}}+x^{\frac{r+3}{2}}-x^r}}, &\emph{if}~r~\emph{is~odd};\\
{\displaystyle \frac{x^{\frac{r}{2}+1}}{1-3x+x^2-2x^{\frac{r}{2}}+x^{\frac{r}{2}+1}+x^r}}, &\emph{if}~r~\emph{is~even},\end{cases}
\end{equation}
and
\begin{equation}\label{gth6e3}
\sum_{n\geq1}\det(1;T_{r+2}^{(r)},T_{r+3}^{(r)},\ldots,T_{n+r+1}^{(r)})x^n=\frac{3x-2x^2-(-x)^{r-2}-(-x)^{r-1}-2(-x)^r}{1-2x+x^2+(-x)^{r-2}+(-x)^{r-1}+(-x)^r}.
\end{equation}
\end{thm}

When $r=3$ in \eqref{gth6e1}, one gets $\sum_{n\geq1}\det(1;T_0,T_2,\ldots,T_{2n-2})x^n=-\frac{x^2(1+x)}{1+3x+2x^3}$ leading to the closed form expression $\det(1;T_0,T_2,\ldots,T_{2n-2})=(-1)^{n-1}(a_{n-2}-a_{n-3})$ for $n\geq 3$ where
$$a_m=\sum_{i=0}^{\left\lfloor\frac{m}{3}\right\rfloor}\binom{m-2i}{i}2^i3^{m-3i}, \qquad m \geq 0,$$
which was not obtained previously.  If $r=3$ in \eqref{gth6e2}, one gets $$\sum_{n\geq1}\det(-1;T_0,T_2,\ldots,T_{2n-2})x^n=\frac{x^2(1-x)}{1-3x-2x^2},$$
which implies the prior explicit formula in this case.  Finally, when $r=3$ in \eqref{gth6e3}, we have $$\sum_{n\geq1}\det(1;T_5,T_6,\ldots,T_{n+4})x^n=\frac{4x-3x^2+2x^3}{1-3x+2x^2-x^3},$$
which leads to the explicit formula for $\det(1;T_5,T_6,\ldots,T_{n+4})$ stated in Theorem \ref{Theorem3} above.

For our next result, we generalize a previous tribonacci determinant identity in terms of a different extension of $T_n$.  Let $S_n^{(r)}$ for $r \geq 3$ odd be defined recursively by $S_n^{(r)}=S_{n-1}^{(r)}+S_{n-\frac{r+1}{2}}^{(r)}+S_{n-r}^{(r)}$ for $n \geq r$, with $S_0^{(r)}=\cdots=S_{r-2}^{(r)}=0$ and $S_{r-1}^{(r)}=1$.  Let $\mathcal{S}_{n}^{(r)}$ denote the set of tilings of length $n$ using pieces of size $1$, $\frac{r+1}{2}$, or $r$.  Note that $S_n^{(r)}=|\mathcal{S}_{n-r+1}^{(r)}|$ for $n \geq r-1$ and that $S_n^{(r)}$ reduces to $T_n$ when $r=3$.

We have the following determinant identity involving $S_n^{(r)}$.
\begin{thm}\label{dgth}
If $r \geq 3$ is odd, then
\begin{equation}
\det\Big(-1;S_{\frac{r-1}{2}}^{(r)},S_{\frac{r+1}{2}}^{(r)},\ldots,S_{n+\frac{r-3}{2}}^{(r)}\Big)=\sum_{i=0}^{\left\lfloor\frac{2n-r-1}{r}\right\rfloor}\binom{2n-r-1-(r-1)i}{i}, \qquad n \geq \small \frac{r+1}{2}\normalsize.
\end{equation}
\end{thm}
\begin{proof}
Let $\mathcal{P}_{n,k}$ denote the set of sequences of tilings $\lambda=(\lambda_1,\ldots,\lambda_k)$ such that $\lambda_i$ uses $s$, $\frac{r+1}{2}$, or $r$ pieces and has length $\ell_i$, where $\ell_1+\cdots+\ell_k=n-\left(\frac{r+1}{2}\right)k$ and $\ell_i \geq 0$ for all $i$.  Let $\mathcal{P}_n=\cup_{k=1}^n\mathcal{P}_n^{(k)}$ and note that the expansion of $\det\big(-1;S_{\frac{r-1}{2}}^{(r)},\ldots,S_{n+\frac{r-3}{2}}^{(r)}\big)$ gives $|\mathcal{P}_n|$, by the definitions.  Let $\mathcal{M}_n=\mathcal{M}_n^{(r)}$ denote the set of marked members of $\mathcal{S}_n^{(r)}$ wherein pieces of length $\frac{r+1}{2}$  may be marked and ending in such a marked piece.  Given $\lambda=(\lambda_1,\lambda_2,\ldots) \in \mathcal{P}_n$, let $\lambda' \in \mathcal{M}_n$ be obtained from $\lambda$ by concatenating the $\lambda_i$ and inserting a marked tile of length $\frac{r+1}{2}$ directly after each $\lambda_i$, including the last.  The mapping $\lambda \mapsto \lambda'$ is a bijection and thus $|\mathcal{M}_n|=|\mathcal{P}_n|$.

Since members of $\mathcal{M}_n$ are synonymous with tilings in $\mathcal{S}_{n-\frac{r+1}{2}}^{(r)}$ in which the $\frac{r+1}{2}$ tiles may be marked, the set $\mathcal{M}_n$ has the same cardinality as the subset of $\mathcal{T}_{2n-r-1}^{(r)}$ whose members contain no dominos, which we will denote by $\widetilde{\mathcal{T}}_{2n-r-1}^{(r)}$.  To realize this, first note that $r$ odd implies that members of $\widetilde{\mathcal{T}}_{2n-r-1}^{(r)}$ must contain an even number of tiles altogether.  Thus, one may group each adjacent pair of consecutive tiles and make the following replacements:  (i) $ss$ by $s$, (ii) $rr$ by $r$, (iii) $sr$ by $\frac{r+1}{2}$, (iv) $rs$ by $\left(\frac{r+1}{2}\right)'$ (indicating a marked tile of length $\frac{r+1}{2}$).  This yields a bijection between $\widetilde{\mathcal{T}}_{2n-r-1}^{(r)}$ and $\mathcal{M}_n$ and thus $\det\big(-1;S_{\frac{r-1}{2}}^{(r)},\ldots,S_{n+\frac{r-3}{2}}\big)=|\widetilde{\mathcal{T}}_{2n-r-1}^{(r)}|$, which implies the result.
\end{proof}

To extend the first formula in Theorem \ref{Theorem3} above, we need to consider the generalized Fibonacci numbers $F_n^{(r)}$ defined recursively by $F_n^{(r)}=F_{n-1}^{(r)}+F_{n-2}^{(r)}+\cdots+F_{n-r}^{(r)}$ for $n \geq r$, with $F_0^{(r)}=F_1^{(r)}=\cdots=F_{r-2}^{(r)}=0$ and $F_{r-1}^{(r)}=1$.  See \cite{Dresden} and \cite[Section~3.4]{BQ}, where they appear in a reparameterized form.  Note that for $n \geq r$, the number $F_n^{(r)}$ counts tilings of length $n-r+1$ where any piece of length up to $r$ may be used.  $F_n^{(r)}$ coincides with $F_{n}$ when $r=2$ and with $T_n$ when $r=3$.  We have the following determinant formula involving $F_n^{(r)}$.
\begin{thm}\label{gfibth}
If $r \geq 2$, then
\begin{equation}\label{gfibthe1}
\det(-1;F_0^{(r)},F_1^{(r)},\ldots,F_{n-1}^{(r)})=\left\lfloor\frac{2^n+2^r-2}{2^{r+1}-2}\right\rfloor, \qquad n \geq 1.
\end{equation}
\end{thm}
\begin{proof}
Given $n \geq r$, let $\mathcal{U}_n=\mathcal{U}_n^{(r)}$ denote the set of tilings of length $n-r$ where pieces of any length up to $r$ may be used and $r$-mino pieces may be marked.  Let $u_n=|\mathcal{U}_n|$ for $n \geq r$, with $u_1=\cdots=u_{r-1}=0$.  By similar reasoning as before, we have that $u_n=\det(-1;F_0^{(r)},\ldots,F_{n-1}^{(r)})$ for all $n \geq 1$.  We first show
\begin{equation}\label{gfibthe2}
u_n=2u_{n-1}+\frac{(-1)^{n+(r-1)\left\lfloor n/r\right\rfloor}}{2}\left((-1)^{\lfloor n/r \rfloor}+(-1)^{\lfloor (n+r-2)/r \rfloor}\right), \qquad n \geq 2.
\end{equation}
Note that \eqref{gfibthe2} is seen to hold for $2 \leq n \leq r+1$, by the stipulated initial values and since $u_r=u_{r+1}=1$, so we may assume $n \geq r+2$.  There are clearly $u_{n-1}$ members of $\mathcal{U}_n$ that end in $s$, so let $\mathcal{U}_n^*$ denote the subset of $\mathcal{U}_n$ whose members do not end in $s$.  To complete the proof of \eqref{gfibthe2}, we define a ``near'' bijection $f$ between $\mathcal{U}_n^*$ and $\mathcal{U}_{n-1}$.

Let $\lambda \in \mathcal{U}_n^*$.  Assume that $\lambda$ contains at least one kind of piece other than a marked $r$-mino and let $z$ denote the rightmost such piece.  If $z$ is a square, then $\lambda \in \mathcal{U}_n^*$ implies $z$ must be followed by a marked $r$-mino.  In this case, we delete $z$ and remove the mark from the marked $r$-mino that directly follows $z$ to obtain $f(\lambda)$.  If $z$ is not a square, then we shorten $z$ by one unit to obtain $f(\lambda)$.  To reverse $f$, consider the position of the rightmost piece that is not a marked $r$-mino and either lengthen it by one unit if it is not an unmarked $r$-mino or change to a marked $r$-mino and insert a square directly prior if it is.  Note that if $n\not\equiv 0,1\text{ (mod
$r$)}$, then $f$ is in fact a bijection between $\mathcal{U}_n^*$ and $\mathcal{U}_{n-1}$.  If $n\equiv 0\text{ (mod $r$)}$, then $f$ fails to be defined for the tiling that consists of a sequence of marked $r$-minos, whence $|\mathcal{U}_n^*|=|\mathcal{U}_{n-1}|+1$ in this case.  If $n\equiv 1\text{ (mod $r$)}$, then $f^{-1}$ is not defined for the same type of tiling of length $n-1$, whence $|\mathcal{U}_n^*|=|\mathcal{U}_{n-1}|-1$.
Therefore, we have $u_n=2u_{n-1}+1$ if $n\equiv 0\text{ (mod $r$)}$, $u_n=2u_{n-1}-1$ if $n\equiv 1\text{ (mod $r$)}$ and $u_n=2u_{n-1}$ otherwise.  Combining these various cases gives recurrence \eqref{gfibthe2}.

We now compute the generating function of $u_n$.  Let $f(x)=\sum_{n\geq 1}u_nx^n$.  Multiplying both sides of \eqref{gfibthe2} by $x^n$, summing over $n \geq 2$, and considering cases mod $r$ for $n$ yields
$$f(x)=\frac{x^r(1-x)}{(1-2x)(1-x^r)}=\left(1+\sum_{i\geq 1}2^{i-1}x^i\right)\left(\sum_{i\geq 1}x^{ri}\right).$$
Computing the coefficient of $x^n$ in this convolution gives
$$u_n=\begin{cases}
{\displaystyle 1+2^{r-1}+2^{2r-1}+\cdots+2^{(m-1)r-1}}, &\text{if $n=rm$};\\
{\displaystyle 2^{q-1}+2^{r+q-1}+\cdots+2^{(m-1)r+q-1}}, &\text{if $n=rm+q$, with $1\leq q\leq r-1$}.\end{cases}
$$
Thus, we have for all $n \geq 1$,
$$u_n=\begin{cases}
{\displaystyle \frac{2^{rm-1}+2^{r-1}-1}{2^r-1}}, &\text{if $n=rm$};\\[6pt]
{\displaystyle 2^{q-1}\left(\frac{2^{rm}-1}{2^r-1}\right)}, &\text{if $n=rm+q$, with $1\leq q\leq r-1$}.\end{cases}
$$

The case of \eqref{gfibthe1} when $n$ is divisible by $r$ now follows immediately from the first case of the last formula. On the other hand, if $n=rm+q$, then $2^{r+1} \equiv 2$ \text{(mod $(2^{r+1}-2)$)} implies
$$2^n=2^{rm+q}\equiv 2^q \text{~(mod $(2^{r+1}-2)$)},$$
and thus
$$2^{q-1}\left(\frac{2^{rm}-1}{2^r-1}\right)=\frac{2^n-2^q}{2^{r+1}-2}=\left\lfloor\frac{2^n}{2^{r+1}-2}\right\rfloor=\left\lfloor\frac{2^n+2^r-2}{2^{r+1}-2}\right\rfloor,$$
as $2^q+2^r-2<2^{r+1}-2$ since $1 \leq q \leq r-1$.  This yields formula \eqref{gfibthe1} in the case when $n$ is not divisible by $r$, which completes the proof.
\end{proof}

Note that the $r=2$ case of \eqref{gfibthe1} gives $\det(-1;F_0,F_1,\ldots,F_{n-1})=\left\lfloor \frac{2^n+2}{6}\right\rfloor$ for $n \geq 1$, with \eqref{gfibthe1} reducing to the first formula in Theorem \ref{Theorem3} when $r=3$.\medskip

\noindent\emph{Remark:} The arguments used to establish \eqref{gth1e1} and \eqref{gth1e3} above show further for $r \geq 2$ that
\begin{equation}
\det(1;F_0^{(r)},F_1^{(r)},\ldots,F_{n-1}^{(r)})=(-1)^{n-1}F_{n-2}^{(r-1)}, \qquad n \geq r-1,
\end{equation}
and
\begin{equation}
\det(1;F_{r-1}^{(r)},F_r^{(r)},\ldots,F_{n+r-2}^{(r)})=(-1)^{n-1}Q_{n+r-1}^{(r)}, \qquad n \geq 1,
\end{equation}
where $Q_n^{(r)}=Q_{n-2}^{(r)}+Q_{n-3}^{(r)}+\cdots+Q_{n-r}^{(r)}$ for $n \geq r$, with $Q_0^{(r)}=1$ and $Q_1^{(r)}=\cdots=Q_{r-1}^{(r)}=0$.  Note that $Q_n^{(r)}$ reduces to $P_n$ when $r=3$.

We still seek a combinatorial proof of the penultimate identity in Theorem \ref{Theorem3} above.  We conclude by providing a bijective proof of the underlying recurrence for the second identity in Theorem \ref{Theorem3}.

\begin{thm}
The sequence $a_n=\det(-1;T_0,T_2,\ldots,T_{2n-2})$ satisfies the recurrence $a_n=3a_{n-1}+2a_{n-2}$ for $n \geq 4$, with $a_2=1$ and $a_3=2$.
\end{thm}
\begin{proof} Let $r$ stand here for a $4$-mino and let us refer to a $4$-mino as \emph{even} if its final section corresponds to an even-numbered position.  Let $\mathcal{V}_n$ denote the set of tilings of length $2n$ that use pieces from $\{s,d,t,r\}$ and end in $r$, where all $r$ pieces are even.  Since $T_{2i-2}$ enumerates all tribonacci tilings of length $2i-4$, it is seen upon considering the number of $4$-minos that $\det(-1;T_0,\ldots,T_{2n-2})$ gives the cardinality of $\mathcal{V}_n$.  Members of $\mathcal{V}_n$ may be regarded as tilings of length $2n-4$ which use $\{s,d,t,r\}$ such that all $r$ are even. Let $a_n=|\mathcal{V}_n|$ for $n \geq 2$.  We establish the recurrence for $a_n$ where $n \geq 4$, the initial conditions being easily verified. First note that there are clearly $2a_{n-1}$ members of $\mathcal{V}_n$ ending in $d$ or $ss$, and $2a_{n-2}$ that end in $r$ or $ts$.  Since tilings in $\mathcal{V}_n$ cannot end in $rs$, to complete the proof, we must show that there are $a_{n-1}$ tilings  that end in $ds$ or $t$.  Consider replacing, within members of $\mathcal{V}_n$ that end in $t$, the final $t$ with $s$.  Let $\widetilde{\mathcal{V}}_n$ and $\mathcal{V}_n^*$  denote the subsets of $\mathcal{V}_n$ whose members end in $ds$ or do not end in $s$, respectively.  Then we can complete the proof by defining a bijection between $\widetilde{\mathcal{V}}_n$ and $\mathcal{V}_{n-1}^*$.

In order to do so, first note that $\lambda \in \widetilde{\mathcal{V}}_n$ implies $\lambda=\alpha sd^is$ or $\lambda=\alpha td^is$, where $i \geq 1$ and $\alpha$ is possibly empty.  Observe no other forms for $\lambda$ are possible since an $r$ cannot appear between the rightmost two tiles of odd length within any member of $\mathcal{V}_n$.  To define the bijection, we treat separately the $i=1$, $i=2$, and $i\geq3$ cases as follows:
\begin{itemize}
\item $\lambda=\alpha sds \rightarrow \lambda'=\alpha d$,
\item $\lambda=\alpha tds \rightarrow \lambda'=\alpha st$,
\item $\lambda=\alpha sd^2s \rightarrow \lambda'=\alpha r$,
\item $\lambda=\alpha td^2s \rightarrow \lambda'=\alpha sdt$,
\item $\lambda=\alpha sd^is,~i \geq 3~ \rightarrow \lambda'=\alpha td^{i-3}t$,
\item $\lambda=\alpha td^is,~i \geq 3~ \rightarrow \lambda'=\alpha sd^{i-1}t$.
\end{itemize}
Considering the various cases, one may verify that the mapping $\lambda \mapsto \lambda'$ furnishes the desired bijection between $\widetilde{\mathcal{V}}_n$ and $\mathcal{V}_{n-1}^*$.
\end{proof}

\end{document}